\documentclass[10pt]{amsart}
\usepackage{times,amsfonts,amsmath,amstext,amsbsy,amssymb,
  amsopn,amsthm,upref,eucal, amscd}
\usepackage[T1]{fontenc}
\usepackage{color}
\usepackage[pdftex]{graphicx}

\newtheorem{theorem}{Theorem}[section]
\newtheorem{lemma}[theorem]{Lemma}
\newtheorem{corollary}[theorem]{Corollary}
\newtheorem{proposition}[theorem]{Proposition}

\numberwithin{equation}{section}

\theoremstyle{definition}
\newtheorem{definition}[theorem]{Definition}

\newtheorem{remark}[theorem]{Remark}

\def\leq{\leqslant }
\def\geq{\geqslant}

\def\F{\mathcal{F}}
\def\H{\mathcal{H}}
\def\B{\mathcal{B}}

\begin{document}

\title[Non-conical strictly convex divisible sets]{Simplicity of Lyapunov spectra and boundaries of non-conical strictly convex divisible sets}

\author{Patrick Foulon}
\address{Patrick Foulon: Aix-Marseille Universit\'e, CNRS, UMR 822, 163 avenue de Luminy, 
13288 Marseille cedex 9, France.}
\email{foulon@cirm-math.fr}
\urladdr{https://test.i2m.univ-amu.fr/perso/patrick.foulon/} 

\author{Pascal Hubert}
\address{Pascal Hubert: Aix Marseille Universit\'e, CNRS, Centrale Marseille, Institut de Math\'ematiques de Marseille, I2M - UMR 7373, 13453 Marseille, France.}
\email{pascal.hubert@univ-amu.fr}
%\urladdr{??????}

\author{Carlos Matheus}
\address{Carlos Matheus: Centre de Math\'ematiques Laurent Schwartz, CNRS (UMR 7640), \'Ecole Polytechnique, 91128 Palaiseau, France.}
\email{carlos.matheus@math.cnrs.fr}
\urladdr{http://carlos.matheus.perso.math.cnrs.fr} 

\date{\today}

\begin{abstract}
Let $\Omega$ be a strictly convex divisible subset of the $n$-dimensional real projective space which is not an ellipsoid. Even though $\partial\Omega$ is not $C^2$, Benoist showed that it is $C^{1+\alpha}$ for some $\alpha>0$, and Crampon established that $\partial\Omega$ actually possesses a sort of anisotropic H\"older regularity -- described by a list $\alpha_1\leq\dots\leq\alpha_{n-1}$ of positive real numbers -- at almost all of its points. In this article, we show that $\partial\Omega$ is maximally anisotropic in the sense that this list of approximate regularities of $\partial\Omega$ does not contain repetitions. This result is a consequence of the simplicity of the Lyapunov spectrum of the Hilbert geodesic flow for every equilibrium measure associated to a H\"older potential.
\end{abstract}
\maketitle
%\pageheight{8.5truein}
%\pagewidth{6.5truein}

%\tableofcontents

\section{Introduction} 

Let $P(\mathbb{R}^{n+1})$ be the real projective space. The group of projective transformations of $P(\mathbb{R}^{n+1})$ is  $PSL(n+1,\mathbb{R})$. A strictly convex open subset $\Omega\subset P(\mathbb{R}^{n+1})$ is \emph{divisible} whenever $\Omega$ is preserved by a torsionless discrete subgroup $\Gamma\subset PSL(n+1,\mathbb{R})$ such that the quotient manifold $M^n=\Omega/\Gamma$ is compact. If $\Omega$ is an ellipsoid, then the quotients $\Omega/\Gamma$ are actually hyperbolic manifolds. On the other hand, a non-conical strictly convex divisible $\Omega$ is necessarily an intricate object: indeed, even though Benoist \cite{B1} proved that the boundary of $\Omega$ is $C^{1+\alpha}$ for some $\alpha > 0$, it is known that $\partial\Omega$ is not $C^2$. Moreover, the geodesic flow $\varphi_t$ for the Hilbert metric is an Anosov flow (see \cite{B1}). In this note, we prove: 

\begin{theorem}\label{th-main} If $\Omega$ is not an ellipsoid, then the Lyapunov spectrum of the geodesic flow $\varphi_t$ with respect to any equilibrium state of $\varphi_t$ derived from a H\"older continuous potential is simple.
\end{theorem} 

A consequence of this result says that the boundary of a non-conical strictly convex divisible set has a complicated infinitesimal behavior at almost all of its points. 

\begin{corollary}\label{t.A} The boundary of a strictly convex divisible set in $P(\mathbb{R}^{n+1})$ which is not an ellipsoid bends infinitesimally according to $n-1$ distinct rates at Lebesgue almost point. More concretely, let $\Omega\subset P(\mathbb{R}^{n+1})$ be a strictly convex divisible set. If $\Omega$ is not an ellipsoid, then there are $n-1$ real numbers $0<\alpha_1<\dots<\alpha_{n-1}$ and a subset $\mathcal{R}\subset\partial\Omega$ of full Lebesgue measure on $\partial\Omega$ with the following property. For each $y\in\mathcal{R}$, one can find a decomposition $T_y\partial\Omega = F_1\oplus\dots\oplus F_{n-1}$ such that the germ of function $f:T_y\partial\Omega\to(\mathbb{R},0)$ whose graph $\{u+f(u)n(y)\}$ (where $n(y)$ is a normal vector to $\partial\Omega$ at $y$) locally describes $\partial\Omega$ satisfies 
$$\lim\limits_{h\to 0}\frac{\log((f(h v_i)+f(-h v_i))/2)}{\log |h|} = \alpha_i$$ 
for any $v_i\in F_i\setminus\{0\}$, $1\leq i\leq n-1$. 
\end{corollary} 

The proof of these results relies on the features of a locally constant cocycle over an Anosov geodesic flow. More concretely, after reviewing the basic theory of convex divisible sets in Section \ref{s.preliminaries} below, we recall in Section \ref{s.Foulon} the relation  between the real numbers $\alpha_i$ in the statement of Theorem \ref{t.A} and the Lyapunov exponents of the so-called Foulon's parallel transport cocycle, and we complete in Section \ref{s.main} the proof of Theorem \ref{th-main} and Corollary \ref{t.A} by establishing the simplicity of the Lyapunov exponents of Foulon's cocycle associated to non-conical strictly convex divisible sets. The main idea in the proof is that the Hilbert flow is locally projectively  flat. It implies that the analysis is reduced to the action of the monodromy group $\Gamma$ acting along the geodesics. 

\subsection*{Acknowledgments} The third author is grateful to CIRM's staff for their immense hospitality during the initial discussions at the origin of this paper. 

\section{Preliminaries}\label{s.preliminaries} 

\subsection{Flat real projective structures} Let $M$ be a closed $n$-dimensional manifold equipped with a flat real projective structure determined by an atlas of charts with values in $P(\mathbb{R}^{n+1})$ whose changes of coordinates belong to $SL(n+1,\mathbb{R})$. 

\begin{figure}[h!]
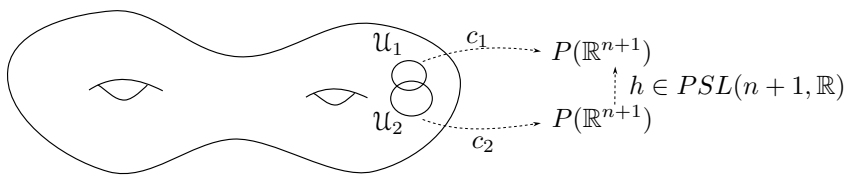\caption{Real projective structures.}
\end{figure}

In this context, we have a holonomy representation $\rho:\pi_1(M)\to SL(n+1,\mathbb{R})$ of the fundamental group $\pi_1(M)$ of $M$, and a developing map $D:\widetilde{M}\to P(\mathbb{R}^{n+1})$ from an universal cover $\widetilde{M}$ of $M$ to the projective space $P(\mathbb{R}^{n+1})$. 

The image $\Omega=D(\widetilde{M})$ of a developing map is an open subset of $P(\mathbb{R}^{n+1})$ which is divisible by the torsionless discrete subgroup $\Gamma=\rho(\pi_1(M))$ of $SL(n+1,\mathbb{R})$ in the sense that the quotient $\Omega/\Gamma\simeq M$ is compact. Note that the natural actions of $\Gamma$ preserve the cone $\mathcal{C}\subset\mathbb{R}^{n+1}$ spanned by $\Omega$ and the boundary $\partial\Omega\subset P(\mathbb{R}^{n+1})$ of $\Omega$. 

We say that a flat real projective structure of $M$ is convex whenever  $\Omega$ is convex. Furthermore, a flat real projective structure of $M$ is properly convex, resp. strictly convex, whenever $\Omega$ is a convex, resp. strictly convex, subset of a compact convex set included in an affine chart of $P(\mathbb{R}^{n+1})$. 

\subsection{Hilbert metrics and Anosov flows} The Hilbert metric of a properly convex open subset $\Omega\subset P(\mathbb{R}^{n+1})$ not containing a line can be defined in terms of cross-ratios as follows. Given distinct points $a, b\in\Omega$, the oriented line $\ell$ from $a$ to $b$ intersects $\partial\Omega$ at two points $a^-, b^+\in\partial\Omega$ labelled in such a way $a^-$, $a$, $b$ and $b^+$ appear in this order in $\ell$. The Hilbert distance $d_{\Omega}(a,b)$ between $a$ and $b$ is 
$$d_{\Omega}(a,b) = \frac{1}{2}\log\left(\frac{|b^+-a|}{|b^+-b|}\cdot\frac{|b-a^-|}{|a-a^-|}\right),$$ 
where $|.|$ is a Euclidean norm in an affine chart containing $\Omega$.  

The intersections of $\Omega$ with projective lines are geodesics of the Hilbert metric $d_{\Omega}$. Moreover, these straight line segments account for all geodesics of $d_{\Omega}$ whenever $\Omega$ is strictly convex: in particular, there is an unique geodesic of the Hilbert metric passing through a given pair of distinct points of $\Omega$ when $\Omega$ is strictly convex. 

\begin{remark} If $\Omega\subset P(\mathbb{R}^{n+1})$ is a conic, i.e., an ellipsoid, then its Hilbert metric coincides with the usual hyperbolic metric.   
\end{remark} 

The Hilbert metric $d_{\Omega}$ is invariant under the projective actions of the elements $g\in SL(n+1,\mathbb{R})$ preserving the cone $\mathcal{C}$ spanned by $\Omega$. In particular, a strictly convex, flat, real projective structure on a closed manifold $M$ allows to equip it with a Hilbert metric. 

In general, the Hilbert metric $d_{\Omega}$ is a Finsler but not Riemannian. More concretely, Benoist \cite{B1} showed that a properly convex open $\Omega\subset P(\mathbb{R}^{n+1})$ which is divisible by a torsionless discrete subgroup $\Gamma\subset SL(n+1,\mathbb{R})$ is strictly convex if and only if $\partial\Omega$ is $C^{1+\kappa}$ for some $\kappa > 0$. In particular, the Hilbert metric $d_{\Omega}$ of a strictly convex, open, divisible subset $\Omega\subset P(\mathbb{R}^{n+1})$ comes from the Finsler structure associate to the continuous family of norms 
\begin{equation}
\|v\|_x = \lim_{t \rightarrow 0} \frac{1}{t } d_{\Omega}(x+tv, x) = \frac{1}{2} \left(\frac{1}{|x^+-x|} + \frac{1}{|x-x^-|}\right) |v| 
\end{equation}
on the tangent spaces $T_x\Omega$ to $x\in\Omega$, where $|.|$ is the Euclidean norm and $x^-$ and $x^+$ are the intersection points of $\partial \Omega$ with the line passing through $x$ with direction $v$. On the other hand, if $\Omega$ is not an ellipsoid, then $\partial \Omega$ is not $C^2$ (in fact, the set of points where $\partial \Omega$ is $C^2$ has zero Lebesgue measure) and the Finsler structure above is not Riemannian. Note that the Finsler structure of $\Omega$ is $\Gamma$-equivariant, i.e., for any $g\in \Gamma$, one has $\|v\|_x = \|g.v\|_{g.x}$. In particular, $M$ has a Finsler structure: indeed, given $x\in M$ and $v\in T_xM$, we can define $F(v)\doteq\|\widetilde{v}\|_{\widetilde{x}}$, where $(\widetilde{x},\widetilde{v})$ is any lift of $(x,v)$ to $T\Omega$.  

If $\Omega$ is a strictly convex open subset of $P(\mathbb{R}^{n+1})$, then the geodesics of its Hilbert metric $d_{\Omega}$ define a $C^{\infty}$ one-dimensional foliation. Since this foliation is invariant under the action of any element of $SL(n+1,\mathbb{R})$ preserving $\Omega$, we see that a strictly convex, flat, real projective structure on a closed manifold $M$ leads to a natural $C^{\infty}$ one-dimensional foliation. As it was proved by Benoist \cite{B1}, this foliation consists of the orbits of a topologically mixing Anosov flow $(\varphi_t)_{t\in\mathbb{R}}$ called the \emph{Hilbert flow}, namely, the geodesic flow of the Hilbert metric of $M$. Here, it is worth to notice that this (Anosov) geodesic flow has the same regularity of $\partial\Omega$, that is, $C^{1+\kappa}$ for some $\kappa>0$. Actually, as we recall below, there is an intimate relationship between the dynamics of $(\varphi_t)_{t\in\mathbb{R}}$ and the geometry of $\partial\Omega$. 

\subsection{Lyapunov exponents and geometry of the boundary}\label{ss.Crampon} Let $M$ be a closed manifold with a flat real projective structure induced by a strictly convex open subset $\Omega\subset P(\mathbb{R}^{n+1})$ which is divisible by $\Gamma\subset SL(n+1,\mathbb{R})$. The geodesic flow $\varphi_t$ with respect to Hilbert metric on the homogenous bundle $HM = (TM\setminus\{0\})/\mathbb{R}_+^*$ is Anosov, i.e., its differential preserves a decomposition 
$$THM = E^s\oplus \mathbb{R}\cdot X\oplus E^u$$ 
where the vectors in $E^s$ are uniformly contracted in the future, $X$ is the generator of the geodesic flow, and the vectors in $E^u$ are uniformly contracted in the past. In general, this decomposition can be further refined at Oseledets regular points $(x,v)\in HM$: by definition, this means that we can find Lyapunov exponents $\chi_1^u(x,v)>\dots>\chi_p^u(x,v)>0 > \chi_1^s(x,v) > \dots > \chi_q^s(x,v)$ and two decompositions $E^s = E^s_1(x,v)\oplus \dots\oplus E^s_q(x,v)$ and $E^u = E^u_1(x,v)\oplus \dots\oplus E^u_p(x,v)$ in Oseledets subspaces such that 
$$\lim\limits_{t\to\pm\infty}\frac{1}{t}\log\|D\varphi_t(x,v) Z_k^u\|_{\varphi_t(x,v)} = \chi_k^u(x,v)$$ 
and 
$$\lim\limits_{t\to\pm\infty}\frac{1}{t}\log\|D\varphi_t(x,v) Z_l^s\|_{\varphi_t(x,v)} = \chi_l^s(x,v)$$ 
for all $Z_k^u\in E_k^u(x,v)\setminus\{0\}$, $Z_l^s\in E_l^s(x,v)\setminus\{0\}$, $1\leq k\leq p$, $1\leq l\leq q$. By Oseledets theorem, the set of Oseledets regular points has full mass with respect to any $(\varphi_t)_{t\in\mathbb{R}}$-invariant probability measure. 

Interestingly enough, Crampon \cite{Cr1} proved that the positive Lyapunov exponents of Oseledets regular points are related to the infinitesimal geometry of $\partial\Omega$. More precisely, if $\chi_1(x,v)>\dots>\chi_p(x,v)>0$ are the positive Lyapunov exponents of the Oseledets regular point $(x,v)\in THM$ and $x^+\in\partial\Omega$ the intersection between $\partial\Omega$ and the oriented line determined any lift of $(x,v)$ to $T\Omega$, then we can find a decomposition $T_{x^+}\partial\Omega=F_1\oplus\dots\oplus F_p$ (by ``projection of Oseledets subspaces'') such that the germ of function $f:(T_{x^+}\Omega,0)\to (\mathbb{R},0)$ whose graph $\{u+f(u)n(x^+)\}$ (where $n(x^+)$ is a normal vector to $\partial\Omega$ at $x^+$) describes $\partial\Omega$ nearby $x^+$ has the following property:  
$$\alpha((x,v),w_i):=\lim\limits_{h\to 0}\frac{\log((f(h w_i)+f(-h w_i))/2)}{\log|h|} = 2/\chi_i(x,v)$$ 
for all $w_i\in F_i\setminus\{0\}$, $1\leq i\leq p$. In other terms, the infinitesimal bending of $\partial\Omega$ at $x^+$ in different directions is dictated by the list of positive Lyapunov exponents along any geodesic ray converging to $x^+$ in the future. 

\subsection{Hopf parametrization of equilibrium states}\label{ss.SRB-BM} Let $m$ be an invariant probability measure of the geodesic (Anosov) flow $(\varphi_t)_{t\in\mathbb{R}}$ of the Hilbert metric on a flat real projective manifold $M$ attached to a strictly convex open subset $\Omega\subset P(\mathbb{R}^{n+1})$ which is divisible by $\Gamma\subset SL(n+1,\mathbb{R})$. It is possible to show that $m$ induces a $\Gamma$-invariant Radon measure on the space $\partial^{(2)}\Omega=\{(x^-,x^+)\in\partial\Omega\times\partial\Omega: x^-\neq x^+\}$ of oriented geodesics of the Hilbert metric of $\Omega$. Furthermore, one can prove that any equilibrium state $\mu_f$ of $(\varphi_t)_{t\in\mathbb{R}}$ with respect to a H\"older potential $f:HM\to\mathbb{R}$ corresponds to a $\Gamma$-invariant Radon measure with a \emph{product} structure $\nu^s\otimes \nu^u$ on $\partial^{(2)}\Omega$. Among the most well-known equilibrium states, one finds the Bowen--Margulis measure $\mu_{BM}$ associated to the trivial potential $f\equiv 0$ and the Sinai--Ruelle--Bowen (SRB) measure $\mu_{SRB}$ associated to the potential $f=\frac{d}{dt}(\log\det D\varphi_t|_{E^u})|_{t=0}$. For later reference, let us recall that the product measure $\nu^s_{SRB}\otimes\nu^u_{SRB}$ on $\partial^{(2)}\Omega$ associated to $\mu_{SRB}$ has the property that $\nu^u_{SRB}$ is absolutely continuous with respect to the Lebesgue measure on $\partial\Omega$.  

\section{Tangent dynamics to geodesic flows of convex divisible sets}\label{s.Foulon} 

\subsection{Parallel transport for  a second order differential equation }
The generator $X =\frac{d}{dt} \varphi_t\vert _{t=0}$ of the Hilbert flow is a second order differential equation on the manifold $M$, that is, a nowhere vanishing vector field on the bundle $ \pi : HM\rightarrow M$ such that, for any $ x\in M, [v] \in H_xM$, we have $ [ d\pi(X([v])]= [v] $. When such a second order differential equation is sufficiently smooth, one can attach to it a geometric structure (extending some of the classical tools known for Riemannian geodesic flows) which is well-adapted to the dynamics of the induced flow: cf. \cite{Fo1}. In the sequel, we briefly recall some of  these tools in full generality in the smooth case. 

In general, there exists a natural decomposition  
\begin{equation}\label{e.vert-hor-decomp}
THM = VHM \oplus \mathbb{R}\cdot X\oplus h_XHM
\end{equation}
with the corresponding projections
$$
Id_{THM} = p^X_v + p^X_X + p^X_H,  
$$
where $VHM= \ker d\pi$ is the vertical distribution and  $ h_XHM$ is the so-called horizontal distribution (defined below).\footnote{In the Riemannian case, one usually consider $ hor_X= \mathbb{R}\cdot X\oplus h_XHM$ as horizontal distribution.} To understand this splitting together with associated invariants, the key ingredient is the verticality lemma for second order differential equation (cf. \cite[Thm. II.1]{Fo1}).  
\begin{lemma} Let  $X$ be a second order differential equation on $HM$. If $Y_1 , Y_2$ are vertical vector fields defined on a neighborhood of $z\in HM$ such that 
$$a.X(z) + Y_1(z) + b .[X,Y_2(z)]= 0$$ 
for two real numbers $a$ and $b$, then 
$$ b.Y_2(z) =0, \quad  a=0 .$$
In particular, the Lie bracket of a non vertical vector field with a second order differential equation is not in $VHM \oplus \mathbb{R}\cdot X$ at non singular points.
\end {lemma}

This lemma permits to introduce the notion of vertical endomorphism. 
\begin{definition}
The fiberwise \emph{vertical endomorphism} $v_X : THM \circlearrowleft$ associated  to a second order differential equation $X$ is the field of $C^\infty$ linear endomorphism defined by the relations:
\begin{enumerate}
\item  $v_X(X)=0$; 
\item $ v_X(Y)=0$ and $v_X([X,Y])=-Y $ for any smooth vertical vector field $Y$.  
\end{enumerate}
\end{definition}
The horizontal operator is obtained from the flow transportation of the vertical bundle. 
 \begin{definition} 
 The \emph{horizontal endomorphism} $ H_X : VHM \rightarrow THM$ is defined by 
 \begin{itemize}
 \item $H_X(Y)= -[X,Y]  - \frac{1}{2} v_X([X,[X,Y]])$ for any smooth vertical vector field $Y$;   
\item it is a zero order differential operator of maximal rank. 
\end{itemize}
The \emph{horizontal distribution} is the vector bundle 
$$h_{X}HM  \doteq H_X (VHM)$$ 
of the same dimension as the vertical bundle.
\end{definition}
This completes the description of the splitting $THM = VHM \oplus \mathbb{R}\cdot X\oplus h_XHM$. For later reference, note that:  
\begin{itemize}
\item $v_X \circ H_X = Id|_{VHM} $, 
\item $ H_X \circ v_X|_{h_{X}HM} = Id|_{h_{X}HM} $.
\end{itemize} 
Cf. Figure \ref{fig.2} below. 
\begin{remark} Here (and in the sequel), we use Lie brackets along the flow and, thus, some derivatives could be involved. Nonetheless, one can check\footnote{This can be done by using any $C^{\infty}$ local function $f$, replacing $Y$ by $f Y$, and checking that no derivatives of $f$ remain.} that the expressions above do not rely on derivatives and, hence, are tensors. Similarly, one takes care of all relevant higher rank Lie derivatives. 
\end{remark}

\begin{figure}[h!]
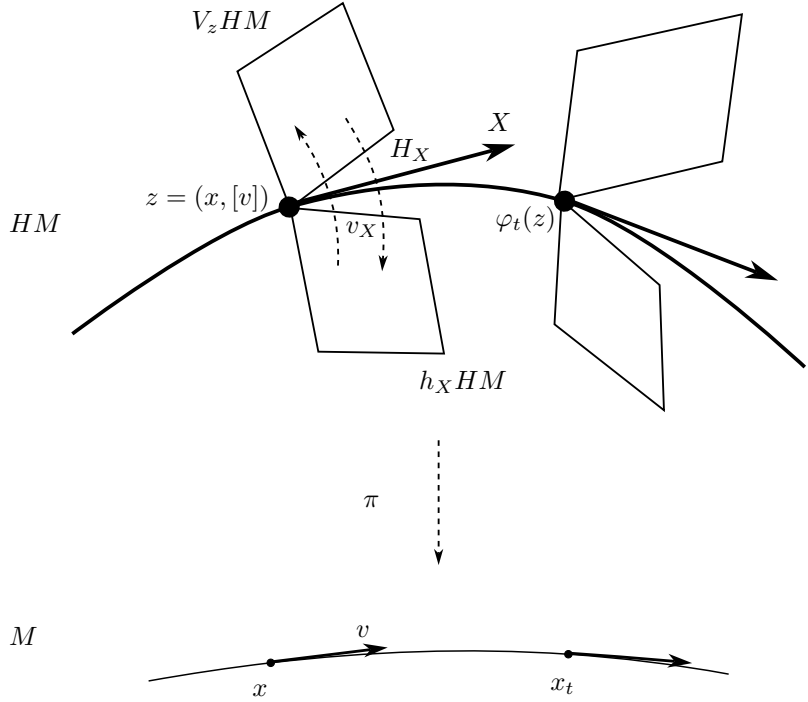\caption{Vertical and horizontal endomorphisms.}\label{fig.2}
\end{figure}

\subsubsection{Dynamical derivative and curvature} 
The action of the flow on the splitting \eqref{e.vert-hor-decomp} produces several tensors and a  kind of covariant derivative along  flow lines using a technique due to Frobenius:  
\begin{definition} The \emph{dynamical derivative} $D_X$ of a second order differential equation $X$ is a first order linear differential operator\footnote{I.e., for any smooth vector field $\xi$ and $ f \in C^1 (HM)$, one has 
$ D_X (f .\xi) = f.D_X (\xi) + L_Xf .\xi $.} on the sections of $THM$ that preserves the splitting \eqref{e.vert-hor-decomp} which is uniquely defined by the relations:  
\begin{enumerate}
\item $D_X (Y)= p^X_v ([X,Y])$ for any smooth vertical vector field $Y$;   
  \item $[D_ X, H_X] = 0$;  
     \item  $ D_X (X) =0$. 
      \end{enumerate}
\end{definition}
\begin{remark} From the definition of the horizontal distribution, we deduce that $ D_X(Y)= - \frac{1}{2} v_X([X,[X,Y]])$. 
\end{remark}
\begin{definition}  A vector field $\xi $ is \emph{parallel} if $ D_X(\xi)=0$. Similarly, a distribution is \emph{parallel} if the dynamical derivative of any of its smooth section remains in the distribution: this is clearly the case for each distribution of the splitting \eqref{e.vert-hor-decomp}. Finally, a tensor is \emph{parallel} if it commutes with $D_X$.
\end{definition}

\begin{definition} 
The \emph{Jacobi endomorphism} or \emph{dynamical curvature} is the $0$-differential linear operator (tensor) defined by the relations: 
\begin{enumerate}
\item $ R_X( Y) = p^X_v ([X,H_X(Y)]) $  for any smooth vertical vector field $Y$; 
\item $[R_X, H_X]=0$; 
\item $R_X(X) = 0$. 
\end{enumerate}
\end{definition}

\begin{remark}\label{r.contact}
As usual, the Jacobi curvature endomorphism measures the flow deviation of the full horizontal distribution in the sense that, for any $\xi\in  hor_X := \mathbb{R}\cdot X\oplus h_XHM $, one has $p^X_v( [X,\xi]) = R_X(v_X(\xi))$. Also, the dynamical derivative is the inner part of the Lie bracket with the flow with respect to the splitting \eqref{e.vert-hor-decomp}: for instance, $ [X,Y] = D_X (Y) - H_X(Y) $. Furthermore, all these tensors and operators are designed to commute with the vertical and horizontal endomorphisms, i.e
 \begin{itemize}
\item{}
$ [ v_X, D_X] =0 \ , \ [ v_X, R_X] =0 $; 
\item{}
$ [ H_X, D_X] =0  \ ,  \ [ v_X, R_X] =0 $.
\end{itemize}
Finally, the distribution $VHM \oplus h_XHM$ might not be flow invariant in general: the lack of flow invariance of this distribution is measured by the one-form   
$$ \beta (h) \doteq p^X_X([X, h])$$ 
on $h_XHM$. 
\end{remark}

\subsubsection{Jacobi equation}
For second order differential equations, it is possible to get an analog of what is known as Jacobi equation extending the usual notion in the Riemannian case. The point is to study infinitesimally the flow transport of  a vector along an orbit: given $z \in HM$ and $\xi_0 \in T_zHM$, we consider $ \xi( \varphi_t(z)) = d\varphi_t (\xi_0)$ which, by definition, commutes with the generator of the flow along the orbit. Using the splitting \eqref{e.vert-hor-decomp}, we can set 
$ \xi = Y + h + \lambda. X$ and then write the commutation relation
$ [X,\xi] = 0 $ as 
\begin{enumerate}
\item 
$ D_X( Y) = - R^X(v_X(h))$; 
\item 
$D_X(h) =  H_X(Y)$; 
\item 
$L_X \lambda = -\beta (h)$. 
\end{enumerate}
Observe that (1) and (2) above  may be combined into the so-called \emph{Jacobi equation}  
$$ D_X D_X(h) + R_X(h) =0.$$

\subsection{Time changes}
For our purposes, it is important to consider the effects of a time change on all of these invariants. Recall that a second order differential equation $X$ has the  same orbits up to time change as $X_0$ near a point $z$ if there exists a neighborhood $U_z$ of $z$ and a positive smooth function $m$ such that 
$$ X =m. X_0 $$
The next proposition relates the geometric structures of $X$ and $X_0$.
 
\begin{proposition}\label{p.time-change} The splitting \eqref{e.vert-hor-decomp} of  $ X=m. X_0$ is related to the splitting of $X_0$ via: 
\begin{enumerate}
\item[(1)] 
$ v_X = \frac{1}{m } . v_{X_0}$
\item[(2)] 
$ H_X (Y) = m . H_{X_0}(Y) + L_Y m . X_0 + \frac{1}{2}L_{X_0} m . Y $
\end{enumerate}
On the vertical bundle $VHM$, the dynamical derivatives and curvatures of $X$ and $X_0$ are linked by the formulas: 
\begin{enumerate} 
\item[(3)] 
$ D_X  = m. D_{X_0} + \frac{1}{2} L_{X_0}m  . Id_{VHM}$
\item[(4)] 
$ R_X = m^2 . R_{X_0} + \lbrack \frac{1}{2}mL^2_{X_0}m - \frac{1}{4}(L_{X_0}m)^2 \rbrack.  Id _{VHM}$
\end{enumerate}
\end{proposition}

\subsection{Projectively flat second order differential equation}
A crucial fact about projective convex structures is that, by definition, their geodesic flows are \emph{locally projectively flat}: recall that a second order differential equation is locally projectively if any $p\in M$ possesses a neighborhood $U_p$  homeomorphic to a domain in $\mathbb{R}^n$ such that, in this chart, all pieces of geodesic are straight segments. Such second order differential equations are the simplest as possible: there is no local dynamics, and the complexity can only come from global properties which are driven by the topology of the compact manifold $M$ or, more precisely, by representations of its fundamental group in $PSL(n+1, \mathbb{R})$. 

As we already mentioned in Section \ref{s.preliminaries}, for strictly convex projective structures, we can pass to the universal cover and use the developing map $\mathcal {D} : \tilde M : \rightarrow P(\mathbb{R}^{n+1})$ to find  an affine chart $\mathbb{R}^n$ containing $ \Omega = \mathcal {D}( \tilde M) $. We can also consider the pull-back of the generator of the Hilbert flow which will be  denoted $\tilde X$ (by abuse of notation). Note that the chart comes with an Euclidean structure whose geodesic flow is generated by $X_0$ defined on the homogeneous bundle $H\Omega$: this Euclidean vector field has obviously trivial  tensors $ R_{X_0} =0$ and $\beta _{X_0}=0 $, and the parallel transport with respect to $X_0$ is the usual Euclidean transport. The next proposition (cf. \cite{Cr2}) says that $\tilde X$ and $X_0$ are related by an \emph{explicit time change} whose regularity is determined by the smoothness of $\partial\Omega$ (namely, $C^{1+\alpha}$ for some $\alpha>0$, but not $C^2$ except in the hyperbolic Riemannian case where $\partial\Omega$ is $C^{\infty}$, see \cite{B1}). 

\begin{proposition}\label{p.explicit-change}
One has $ \tilde X =m. X_0 $ where $ m: H\Omega  \rightarrow \mathbb{R}_*^{+}$  is the positive function 
$$ m(x , [v]) = 2 . \frac{ |x^+-x| . |x-x^-|}{|x^+-x^-|}$$
related to the Finsler structure underlying the Hilbert metric of $\Omega$. In particular, $m$ has the same regularity of $\partial\Omega$, i.e., $m \in C_{X_0}^{1+\alpha} (\Omega,\mathbb{R}^{+,*})$ (where this notation means that all derivatives along the flow of $X_0$  are still in $C^{1+\alpha}$).
\end{proposition}

\begin{remark}\label{r.special-chart} 
A change of affine chart by a projective transformation preserves these formulas, the Hilbert distance and the regularity of $m$. A particularly convenient choice of chart \emph{adapted} to a point $ (x_0, [v_0])$ consists of selecting the affine coordinates such that the tangents to $\Omega$ at the points $ x^+_0, x^-_0$ are orthogonal to $[v_0]$ (and parallel among themselves). On the other hand, the function $m$ does \emph{not} descend to $HM$. 
\end{remark}

%picture 

\subsection{Hilbert form}
Crampon \cite[Prop. 3.6]{Cr2} took advantage of the special charts in Remark \ref{r.special-chart} to prove that the Hilbert flow is a sort of contact flow: 
\begin{proposition} For a strictly convex divisible subset $\Omega$, one has $\beta = \beta_{\tilde X} = 0$ or, equivalently, the  distribution $VHM \oplus h_X HM $ is flow invariant (cf. Remark \ref{r.contact}). In particular, the Hilbert form $A$ defined by
$$ A(X)= 1, \quad A_{\vert VHM \oplus h_X HM} =0$$
is a flow invariant 1-form on $THM$ which is H\"older continuous.
\end{proposition} 

\begin{remark} 
\mbox{ }

\begin{itemize}
\item In smooth Riemannian or Finsler settings, the flow invariance of the Hilbert form (which is eventually the vertical differential of the Finsler norm, cf. \cite{Fo1}) is  classical and easily obtained from variational calculus.\\
\item  For strongly convex smooth Finsler metrics,  the distribution $VHM \oplus h_X HM $ is a contact structure invariant by the flow and $A \wedge dA^{n+1}$ is an invariant volume form.\\
\item  The lack of differentiability in our case makes the proof of the flow invariance into a deep consequence of  the projective invariance. \\
\item The non-existence of the exterior differential of the Hilbert form suggests that there is no  flow invariant probability measure which is absolutely continuous with respect to a Lebesgue measure except in the  Riemannian case. This was confirmed by Benoist \cite{B1} using the fact that each periodic orbit carries an obstruction. The vanishing of all these obstructions is possible only if the image of the fundamental group representation is not Zariski dense, i.e., in Riemannian case. 
\item  As shown by Crampon, all the formulas presented above for projectively flat \emph{smooth} second order differential equations extend to this \emph{low regularity} situation.  
  \end{itemize}
 \end{remark}
 
Moreover, Crampon \cite{Cr2} checked that Hilbert geometries have constant negative curvature (among other properties): 
\begin{proposition}\label{p.constant-curvature} Let $X$ be the Hilbert generator of a strictly convex  projective structure. 
\begin{enumerate} 
\item  On  $VHM \oplus h_X HM $,  
$$ R^X \equiv- Id$$
\item Given an affine chart with generator $X_0$, if $Y_0$ is an Euclidean parallel vector field, then the vertical vector field 
$ Y= ( m^{-1/2} . Y_0)$ is $X$-parallel
$$ D_X ( m^{-1/2} . Y_0) =0$$
\item  If $h_0 $ is an Euclidean parallel horizontal vector field, then  
$ Y_0 = V_{X_0} (h_0)$ is also an Euclidean parallel vector field and 
$$ h = m^{1/2}h_0 + L_{X_0} m^{1/2} . Y_0+ 2. L_{Y_0}m^{1/2} . X_0 $$
is a horizontal parallel vector field for $X$ (thanks to Prop. \ref{p.time-change} (1) and (2)).
Here, these formulas are explicitly related to the boundary of $ \Omega$ (via Prop. \ref{p.explicit-change}).
\end{enumerate}
\end{proposition} 

\subsection{Stable and unstable distributions} 

The tangent spaces to the stable and unstable horospheres of the Hilbert flow are given by the following (picture and) result: 

\begin{figure}[h!]
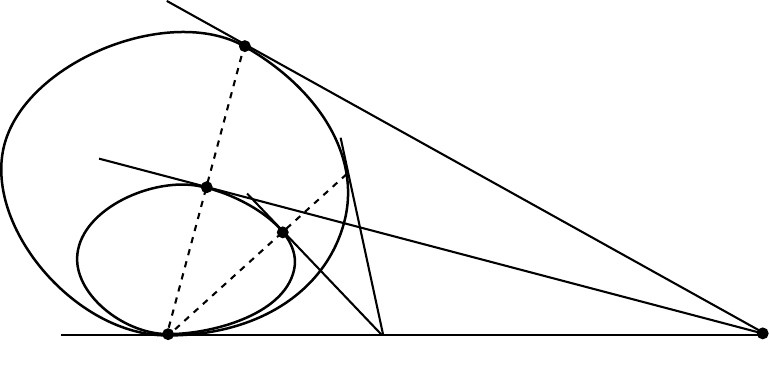\caption{Tangent distributions to stable horospheres $S_{x^+}(x)$. See \cite{B1} for more details on this construction.}
\end{figure}

\begin{proposition} Let $X$ be the geodesic generator of a strictly convex projective structure on a compact manifold. The following distributions are flow invariant and $X$-parallel
$$ E^+ = \lbrace \xi =Y+H_X (Y) , Y \in VHM \rbrace$$
$$ E^-  = \lbrace  \xi  =-Y +H_x(Y) , Y \in VHM \rbrace $$
\end{proposition}

\begin{proof} This a direct application of the facts that $R^X =- Id $ and of the flow invariance of the Hilbert form. For example, if $ \xi \in E^+ $, then  
$$ [X, \xi] = - H_X (Y) + D_X(Y) + D_X (H_X (Y)) + R^X(Y) + \beta (H_X (Y)) .X $$
thus
$$ [X, \xi]  = ( -Y + D_X(Y)) + H_x((-Y + D_X(Y)) $$
lies in $ E^+$. 
\end{proof}

Furthermore, it is possible to check that the action of the Hilbert flow on stable horospheres is described by the following figure: 

\begin{figure}[h!]
%% Creator: Inkscape 1.2.2 (732a01da63, 2022-12-09), www.inkscape.org
%% PDF/EPS/PS + LaTeX output extension by Johan Engelen, 2010
%% Accompanies image file 'figure4.pdf' (pdf, eps, ps)
%%
%% To include the image in your LaTeX document, write
%%   \input{<filename>.pdf_tex}
%%  instead of
%%   \includegraphics{<filename>.pdf}
%% To scale the image, write
%%   \def\svgwidth{<desired width>}
%%   \input{<filename>.pdf_tex}
%%  instead of
%%   \includegraphics[width=<desired width>]{<filename>.pdf}
%%
%% Images with a different path to the parent latex file can
%% be accessed with the `import' package (which may need to be
%% installed) using
%%   \usepackage{import}
%% in the preamble, and then including the image with
%%   \import{<path to file>}{<filename>.pdf_tex}
%% Alternatively, one can specify
%%   \graphicspath{{<path to file>/}}
%% 
%% For more information, please see info/svg-inkscape on CTAN:
%%   http://tug.ctan.org/tex-archive/info/svg-inkscape
%%
\begingroup%
  \makeatletter%
  \providecommand\color[2][]{%
    \errmessage{(Inkscape) Color is used for the text in Inkscape, but the package 'color.sty' is not loaded}%
    \renewcommand\color[2][]{}%
  }%
  \providecommand\transparent[1]{%
    \errmessage{(Inkscape) Transparency is used (non-zero) for the text in Inkscape, but the package 'transparent.sty' is not loaded}%
    \renewcommand\transparent[1]{}%
  }%
  \providecommand\rotatebox[2]{#2}%
  \newcommand*\fsize{\dimexpr\f@size pt\relax}%
  \newcommand*\lineheight[1]{\fontsize{\fsize}{#1\fsize}\selectfont}%
  \ifx\svgwidth\undefined%
    \setlength{\unitlength}{221.12933061bp}%
    \ifx\svgscale\undefined%
      \relax%
    \else%
      \setlength{\unitlength}{\unitlength * \real{\svgscale}}%
    \fi%
  \else%
    \setlength{\unitlength}{\svgwidth}%
  \fi%
  \global\let\svgwidth\undefined%
  \global\let\svgscale\undefined%
  \makeatother%
  \begin{picture}(1,0.73346044)%
    \lineheight{1}%
    \setlength\tabcolsep{0pt}%
    \put(0,0){\includegraphics[width=\unitlength,page=1]{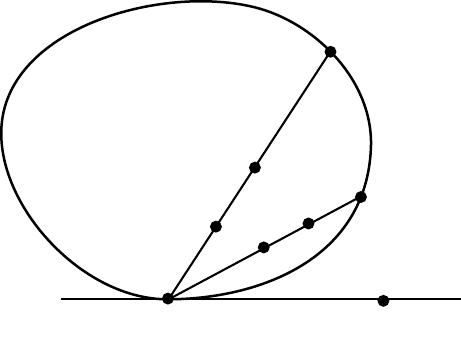}}%
    \put(0.31564472,0.01258192){\color[rgb]{0,0,0}\makebox(0,0)[lt]{\lineheight{1.25}\smash{\begin{tabular}[t]{l}$x^+$\end{tabular}}}}%
    \put(0.37181486,0.24439021){\color[rgb]{0,0,0}\makebox(0,0)[lt]{\lineheight{1.25}\smash{\begin{tabular}[t]{l}$x_t$\end{tabular}}}}%
    \put(0.53611652,0.23051501){\color[rgb]{0,0,0}\makebox(0,0)[lt]{\lineheight{1.25}\smash{\begin{tabular}[t]{l}$y_t$\end{tabular}}}}%
    \put(0.72753855,0.65292903){\color[rgb]{0,0,0}\makebox(0,0)[lt]{\lineheight{1.25}\smash{\begin{tabular}[t]{l}$x^-$\end{tabular}}}}%
    \put(0.80166805,0.28420991){\color[rgb]{0,0,0}\makebox(0,0)[lt]{\lineheight{1.25}\smash{\begin{tabular}[t]{l}$y^-$\end{tabular}}}}%
    \put(0.4557729,0.3754079){\color[rgb]{0,0,0}\makebox(0,0)[lt]{\lineheight{1.25}\smash{\begin{tabular}[t]{l}$x$\end{tabular}}}}%
    \put(0,0){\includegraphics[width=\unitlength,page=2]{figure4.pdf}}%
    \put(0.65247358,0.28408951){\color[rgb]{0,0,0}\makebox(0,0)[lt]{\lineheight{1.25}\smash{\begin{tabular}[t]{l}$y$\end{tabular}}}}%
  \end{picture}%
\endgroup%
\caption{Action of the Hilbert (geodesic) flow on a stable horosphere $S_{x^+}(x)$. See \cite{B1} for more details on this construction.} 
\end{figure}

\subsection{Derivative of the Hilbert flow} \label{sec:derivative-Hilbert-flow}
The Jacobi equation says that we just need to globally understand the parallel transport to control the action of the flow: 
\begin{proposition}
 Given $z =(x, [v])$ 
and $\xi^{\pm}_0 \in E^{\pm}_z$, let   
$ \xi^{\pm}(\varphi_t(z))= d\varphi_t (\xi_0^{\pm})$. Then, 
\begin{equation}
 \xi^{\pm}(\varphi_t(z)) = e^{\pm t} .  \ \eta^{\pm}(\varphi_t(z)),
\end{equation}
where 
\begin{equation}
  \eta^{\pm}= H_X (Y) \pm Y
 \end{equation}
 is an explicit parallel vector field in $E^{\pm}$ (that is, $ Y $ is a $X$-parallel  vertical vector field). 
\end{proposition}
\begin{proof} 
If $Y$ is a parallel vertical vector field, then $ e^{\pm t} Y$ is solution of the Jacobi equation because  
$$  D_X (e^{\pm t} Y) = e^{\pm t}(\pm Y  + D_X(Y)) =e^{\pm t}(\pm Y ) $$ 
and 
$$ D_X D_X (e^{\pm t} Y) + R_X(e^{\pm t} Y)= e^{\pm t}( Y ) - e^{\pm t} Y  =0 .$$
\end{proof} 

\begin{remark} This situation must be compared with the analysis of the derivative of the geodesic flow of Finsler structure on the moduli spaces of translation surfaces induced by the Teichm\"uller metric. In fact, the derivative of this famous geodesic flow is known to be the tensor product of diagonal matrices $\textrm{diag}(e^t, e^{-t})$ by the so-called Kontsevich--Zorich cocycle (corresponding to a parallel transport with respect to the Gauss--Manin connection): see, e.g., Forni's paper \cite{F} for more explanations. 
\end{remark}

\subsection{Foulon's cocycle}\label{ss.Foulon}

In the sequel, the parallel transport will be denoted by $ \tau_t Y_0 = Y(\varphi_t(z))$. Since the tensors $ v_X, H_X$ are H\"older continuous, they are bounded.  Let us choose a norm on $THM$ by declaring that the splitting \eqref{e.vert-hor-decomp} is orthogonal and by writing 
\begin{equation}\label{e.norm-choice-1}
 \vert \xi  \vert = F( d\pi (\xi)) + F ( d\pi( H_X( p^X_v (\xi)) .
 \end{equation}
We will use the same notation for the lifted norm on $TH\Omega$. 
It is also useful to introduce a norm\footnote{Note that it is \emph{not} a norm on $TM$ !} on the vertical bundle $VHM$
\begin{equation}\label{e.norm-choice-2}
 N( Y ) = F ( d\pi( H_X(Y))  
 \end{equation}

So, the Lyapunov exponent of the flow at $\xi^{\pm}_0 \in E^{\pm}_z$ is given by the study of 
$$ \lim_{t \rightarrow \pm \infty }\frac{1}{t} \ln (\vert d \varphi _t \xi _0 \vert) = \pm 1+ \lim_{t \rightarrow \pm \infty }\frac{1}{t} \ln (\vert \tau_t  \xi _0 \vert)$$
Using \eqref{e.norm-choice-1} and \eqref{e.norm-choice-2}, we can express for any $ z\in  HM$ and $\xi^{\pm}_0  \in E_z^{\pm}$
\begin{equation}
\vert d \varphi _t \xi^{\pm}_0 \vert = F( d\pi (\xi^{\pm}_t) + F ( d\pi( H_X( p^X_v (\xi^{\pm}_t)) = 2 e^{\pm t} \lbrace F( d\pi (H_X (Y_t)) \rbrace
 \end{equation}
Furthermore  $Y_t = \tau_t Y_0 = \tau_t( \pm v_X (\xi^{\pm}_0)) = \tau_t (\pm p^X_v(\xi^{\pm}_0))$ is a parallel vertical vector field along the flow orbit of the point $z$. Hence, 
$$ \lim_{t \rightarrow \pm \infty }\frac{1}{t} \ln (\vert d \varphi _t \xi _0 \vert) = \pm 1+ \lim_{t \rightarrow \pm \infty }\frac{1}{t} \ln (N (\tau_t Y _0 ))$$

The outcome of this discussion is that the task of understanding the action of the flow on vectors in the kernel of the Hilbert form is reduced to understand how the parallel transport associated the flow acts on the vertical bundle. For this sake, it is useful to work on the universal cover $\Omega$, choose any affine chart with an Euclidean norm, and to use Proposition \ref{p.constant-curvature} (2). So with the same notation if $Y_w$ is a vertical vector at $w\in H\Omega$ then 
the parallel transport may be expressed using this chart through a trivial cocycle relating it to the Euclidean parallel transport written in that chart. 
\begin{equation}
Y_{\varphi _t (w)} = \tau_t Y_w  = \frac {m^{-1/2}(\varphi _t (w))}  {m^{-1/2}(w) } \ \tau^0_t  Y_z
\end{equation}
\begin {remark} 
\begin{itemize}
\item Any other affine chart will give the ``same'' cocycle. 
\item The parallel transport factor is independent of the transported vector.
\item Using the chart $H\Omega$ may be identified with $\Omega \times S^{n-1}$ and the tangent bundle 
$ TH\Omega $ and the vertical fiber at $ (x,u)$ as the tangent  space to the sphere in direction $ u$ located at $x$ .   The vector  $\tau^0_t  Y_z$ is the same vector as $ Y_z$ but over the point $c_u (t)$ with $c_u (0)=x$
\end{itemize}
\end {remark}

Such an affine chart is not defined on the manifold $M$, so the induced cocycle on the manifold $M$ is no longer  trivial. To understand this, let us choose a compact fundamental domain  $\Delta$. We will take advantage latter that such a domain may be chosen  polyhedral in the affine chart (see, e.g., \cite{Ma}). So, let $ z \in HM$ and $ \tilde z \in H\Delta $ its lift. 
Given $t$ there is a unique $g_t$ in $ \Gamma$ such that $g. \varphi _t (w) \in H\Delta$, where we consider the action of the representation of the fundamental group on the homogenous bundle $H\Omega$. Such an action maps fiber to fiber and so preserve the vertical bundle. The map 
 $ \Theta  : H\Delta \times \mathbb{R}  \rightarrow \Gamma $  
defined by    $ \Theta (z , t)  = g_t$  is not continuous, and  using the equivariance of the Finsler norm  on $\Omega$, of the base projection and of  $H_X$ we get for any vertical vector $Y \in V_wH\Omega$
\begin{equation}
 N(Y) _w=  N(gY) _{g.w} 
 \end{equation} 
 So we can expressed the norm of a vertical vector along a flow orbit 
 \begin{equation}
N( Y_{\varphi _t (z)})  = \frac {m^{-1/2}(\varphi _t (\tilde z))}{m^{-1/2}(\tilde z) }N(g_t ( \tau^0_t  Y_{\tilde z})) 
\end{equation} 
The advantage of considering vectors only in a fixed compact fundamental domain is that the norm is bounded in the affine chart.  \\

We can go a bit further using the fact that all this is independent of the choice of affine chart. Given  an orbit we can  choose an  adapted chart (see  Remark \ref{r.special-chart}). In this case, Proposition \ref{p.constant-curvature} (3) gives along that single orbit that 
$$ H_X (Y) = m^{1/2}H_{X_0} (Y_0)+ L_{X_0} m^{1/2} . Y_0$$
and, \emph{a fortiori},  
$$ N( Y_{\varphi _t (z)})  = \frac {m^{1/2}(\varphi _t (\tilde z))}{m^{1/2}(\tilde z) }\| d\pi( \tau^0_t (H_{X_0}Y_{\tilde z}))\|_ {x_t}$$
The vector $Z_t= d\pi( \tau^0_t (H_{X_0}Y_{\tilde z}))$ is only expressed in terms of the affine chart and with the Euclidean parallel transport which descends to the base manifold  it is the euclidean parallel transport along the chosen geodesic  of the tangent vector $ d\pi  (H_{X_0}Y) \in T_{\tilde x} \Omega$. Thus in a trivialization adapted to the chart $ T\Omega =\Omega \times \mathbb{R}^n$ the vectors field $Z_t$ defined along the geodesic is constant.  
\begin{remark}
\mbox{ }

\begin{itemize}
\item Any horizontal vector is in the kernel of the Hilbert form and projects down with $ d \pi $ to a vector tangent to the horosphere.  So $Z_t$ is tangent to the two horospheres at point $ x_t$ pointing respectively to the chosen points $x^+$ and $x^-$.
\item Notice the exponent in the cocycle has changed in the numerator because we have had to use  euclidean horizontal vector fields to express the norm.
\item The advantage to use parallel horizontal vector fields is that after projection on the base $\Omega$ we will just have to consider the action of the group on the tangent space. instead of the action tangent  space of the fiber of the homogeneous bundle.
\end{itemize}
\end{remark}

\begin{proposition}\label{p.3.21}
\begin{equation} 
N( Y_{\varphi _t (z)})   =  \frac{1}{2} .  \frac {m^{1/2}(\varphi _t (x, [v]))}{m^{1/2}(x, [v]) }.F(x_t, Z_t)\\ 
 \end{equation} 
and, when using the group action on $\Omega$, 
\begin{equation} 
N( Y_{\varphi _t (z)}) = f(z,t) . F(g_tx , g_tZ_t) 
 \end{equation}
 with 
 \begin{equation}\label{e.function-f} 
 f(z,t)  = \frac{1}{2} .  \frac {m^{1/2}(\varphi _t (x, [v]))}{m^{1/2}(x, [v]) }.
 \end{equation}
\end{proposition}
%
%\textcolor{red}{[ To do: justify along the lines of Patrick's email [almost closing up orbits with bounded error terms to almost replicate the computation in Crampon along periodic orbit with translation length $(1/2)\log(|\lambda_1|/|\lambda_{n+1}|)$, etc. ...] the fact that the exponents $\eta_1,\dots,\eta_{n-1}$ relate to the exponents $\lambda_0,\lambda_1,\dots,\lambda_n$ of locally constant cocycle $\rho:\pi_1\to SL(n+1,\mathbb{R})$ via $\eta_i = -1+2\log(|\lambda_0/\lambda_i|/|\lambda_0/\lambda_n|)$, ...] ?????? [Crampon's thesis [http://mikl.crampon.free.fr/these/these$\_$up.pdf] is a good reference...]} 

%In view of the results in \S\ref{ss.Foulon}, this is equivalent to the simplicity of the Lyapunov exponents with respect to $\mu_{SRB}$ of 

%$$ \vert  \vert  Y  \vert  \vert  = \vert  \vert  g(Y)  \vert  \vert $$
%$$ \vert  \vert  Y  \vert  \vert = F ( d\pi( H_X(Y)) $$
%$$ \\  \vert \vert  g_t ( \tau^0_t  Y_{\tilde z}) \vert \vert  = \vert \vert  \tau^0_t  Y_{\tilde z} \vert \vert $$ 

\subsection{Explicit parametrisation of the geodesic flow}

\begin{lemma}
In an affine chart containing $\Omega$, the lift of Hilbert geodesic flow on $H \Omega$ is trivialized as $\varphi_t(x, [v]) = (x_t, [v])$, where 
 $(x_t)_{t\in\mathbb{R}}$ is the geodesic with corresponding endpoints $ (x^-, x^+)$ parametrised by $t\in\mathbb{R}$ in such a way that 
 \begin{equation}
\vert x-x_t \vert =\frac{ \vert x^- - x \vert . \vert x-x^+ \vert .( 1- \exp(-2t))}{\vert x^- - x \vert + \vert x-x^+ \vert . \exp(-2t)}.
 \end{equation}
 \end{lemma}
 
\begin{proof}
This follows from the fact that $d_{\Omega} (x,x_t)= t$: see, e.g., \cite[Lemma 4.5]{Cr2}.
\end{proof} 

We may use this expression of the geodesic flow together with the definition of $m$ in Proposition \ref{p.explicit-change} to compute the function $f$ introduced in \eqref{e.function-f}: 
\begin {lemma}\label{l.3.14}
\begin{equation} 
 f((x,[v]),t)  =  \frac{1}{2} . \frac {\vert x^- - x^+ \vert }
{\vert x^- - x \vert. \exp(t) + \vert x-x^+  \vert . \exp(-t)} \end{equation}
\end{lemma}
\begin{proof} This is a straightforward computation with the previous formulas. 
\end{proof}

\subsection{Lyapunov exponents along periodic orbits}
Periodic orbits of the Hilbert flow were studied by Benoist \cite{B1} and their Lyapunov exponents were computed by Crampon \cite{Cr2}. More concretely, consider a closed orbit $\gamma_g$ of the (Hilbert) geodesic flow associated to a group element $g\in \Gamma$ and denote by $T =  L ^H(g)$ the length of $\gamma_g$ (with respect to the Hilbert metric). After Benoist, the element $g$ is bi-proximal. Furthermore, Crampon showed that if $x$ is a point of $\gamma_g$, then the parallel transport by time $T =  L ^H(g)$ of a vector $ Z \in T_{x}S_{x^+}(x)$ tangent to the stable horosphere at $x$ with point at infinity $x^+$ satisfies: 
 
\begin{proposition} \label{pro:exp-periodic-orbit} Let  $g\in \Gamma$. The eigenvalues of the matrix associated to $ g$ have moduli (possibly repeated with multiplicity) of the form 
  $ (\lambda_0 > \dots\geq \lambda_i \geq \dots > \lambda_n)$. Moreover, if $E_i$ denotes the generalised eigenspace associated to eigenvalues with the same modulus of $\lambda_i$, then the Hilbert parallel transport $ \tau^H_T$ along the associated  closed orbit $ \gamma_g$ for time $T=L^H(g)$ preserves a splitting 
$$ THM = \mathbb{R}X \oplus\bigoplus_{i} T_iHM $$
where   $$ T_iHM= V_iHM \oplus {H_X} (V_iHM) \quad \textrm{ and } \quad d\pi {H_X} (V_iHM) = E_i.$$
%such that for any $X_i \in T_iHM$ 
%$$ \tau^H_T(X_i) = {\lambda_0^{\frac{1}{2}}}{\lambda_n^{\frac{1}{2}}}. \frac{U_i^{-1}.X_i}{\lambda_i}. $$
%\end{proposition}
%\begin{remark}
%In the Riemannian case any $g\in \Gamma$ is such that $ \lambda_0.\lambda_n =1$ and $\lambda_i= 1$ then the parallel transport is an orthogonal transformation and the corresponding Lyapunov  exponents vanish.
%\end{remark}
%\begin{corollary}
Furthermore, the parallel Lyapunov exponent for $X_i \in T_iHM $ is 
$$ \eta_i = 1 + 2 \frac{\log \frac{\lambda_n}{\lambda_i}}{\log{\frac{\lambda_0}{\lambda_n}}} = -1 + 2 \frac{\log \frac{\lambda_0}{\lambda_i}}{\log{\frac{\lambda_0}{\lambda_n}}}.$$
In particular 
$$ -1 <  \eta_i < 1.$$ 
\end{proposition} 

\begin{remark} In the case of a \emph{diagonalisable} element $g\in\Gamma$, an \emph{exact} formula for $\tau_T^H$ is available: see Appendix \ref{a.1} below for more details.  
\end{remark}

%\end{corollary}
%\noindent Proof of proposition 3.25
%We have identified the tangent space to the horosphere at $x$ to the space $ \oplus E_i$ according to (3.15)
%We have just to evaluate $dg^-1(Z_i)$ this is done by computing the differential of the map given in 3.16 and taking its inverse at the point $x_T$ and we obtain for any $Z_i$ 
%$$ dg^{-1}(Z_i)= \frac{\lambda_0. u + \lambda_n .v}{ \lambda_i} U_i ^{-1}. Z_i$$
%As mentioned before with (3.18), (3.19),(3.20) we get
%$$ f((x,[v]), T ) = \frac{1}{2} .  \frac {1}{u. \exp(T) + v . \exp(-T)}$$ 
%and with proposition 3.24
%$$ f((x,[v]), T ) = \frac{1}{2} . \lambda_n^{\frac{1}{2}} . \lambda_0^{\frac{1}{2}} .  \frac {1}{u. \lambda_0+ v . \lambda_n}$$ 

\section{Proof of Theorem \ref{t.A}}\label{s.main}

Let us define the locally constant cocycle over the Hilbert flow $(\varphi_t)_{t\in\mathbb{R}}$ of $M$ induced by the natural representation $\rho:\pi_1(M)\to \Gamma < SL(n+1,\mathbb{R})$.
In plain terms, on the trivial bundle $\hat{E} = TM\times \mathbb{R}^{n+1}$, we have the trivial cocycle $A^t(x).u = (\varphi_t(x), u).$  The group $\Gamma$ acts on $\hat{E}$  by
 $\gamma(x,u) = (\gamma(x), \rho(\gamma).u)$. Since the action of $\Gamma$ and $\varphi_t$ commute, it induces a cocycle on $E = \hat{E}/\Gamma$ which is denote by $\rho$ (in a slight abuse of notation).

\begin{lemma} \label{lem:link cocycles}
Given $\mu$ an ergodic invariant measure by the Hilbert geodesic flow on a compact $n$-dimensional manifold, the Lyapunov exponents of the parallel transport are related to the Lyapunov exponents of the cocycle $\rho$ by the following formula
\begin{equation}
\eta_i = -1 + 2\frac{\ell_0 -\ell_i}{\ell_0 -\ell_{n}}
\end{equation} 
where $\eta_1 \leq \cdots \leq \eta_{n-1}$ are the $(n-1)$ Lyapunov exponents of the parallel transport and $\ell_0  \leq \cdots \leq \ell_{n}$ are the $(n+1)$ Lyapunov exponent of the cocycle $\rho$.
\end{lemma}

\begin{proof}
The result is true for measures supported by periodic orbits by the result of Crampon (cf. Proposition \ref{pro:exp-periodic-orbit} above). Since the Hilbert flow is an Anosov flow, it is also true for every ergodic invariant probability measure because one can approximate Lyapunov exponents of every ergodic invariant measures by the Lyapunov exponents of measures supported by periodic orbits (see \cite[Theorem 1.4]{Ka}).
\end{proof}

 The desired Theorem \ref{th-main} is an immediate consequence of the following result: 

\begin{theorem}\label{t.B} If $\Omega$ is not an ellipsoid, then the Lyapunov spectrum of the cocycle associated to $(\varphi_t,\rho)$ with respect to any equilibrium state of $\varphi_t$ derived from a H\"older continuous potential is simple.
\end{theorem}

\begin{proof}[Proof of Theorem \ref{th-main} assuming Theorem \ref{t.B}] 

We first reduce our study to the Lyapunov exponents of the parallel transport since it only eliminates the redundancy of what happens on the stable and unstable manifold (see Section \ref{sec:derivative-Hilbert-flow}). 
Moreover, by Lemma  \ref{lem:link cocycles}, the simplicity of the Lyapunov spectrum of the parallel transport is equivalent to the simplicity of the Lyapunov spectrum of the cocycle $\rho$. Thus, it is enough to prove Theorem \ref{t.B}.
\end{proof}

\begin{proof}[Proof of Theorem \ref{t.B}]
Since the Hilbert flow $(\varphi_t)_{t\in\mathbb{R}}$ is a topologically mixing Anosov flow (cf. \cite{B1}), the cocycle $(\varphi_t,\rho)$ is coded by a locally constant cocycle over a subshift of finite type, say $(\sigma, \rho)$, where $\sigma$ acts by left-shift on the elements of $\Sigma\subset\mathcal{A}^{\mathbb{Z}}$ (and $\mathcal{A}$ is a finite alphabet). 

As it is explained in \cite[Section 3]{BGMV}, the equilibrium states $\mu$ of $\varphi_t$ with respect to H\"older continuous potentials are coded by measures $\nu$ with a \emph{local product structure}. By definition, the Lyapunov exponents of $(\varphi_t,\rho)$ with respect to $\mu$ are simple if and only if the Lyapunov exponents of $(\sigma, \rho)$ with respect to $\nu$ are simple. In particular, the desired statement can be deduced from the simplicity criterion of Bonatti and Viana \cite{BV} (see also \cite{AV}) provided we ensure that $(\sigma, \rho)$ is \emph{typical} in the sense that: 
\begin{enumerate}
\item there is a periodic point $p_{\theta}\in\Sigma$ of $\sigma$ associated to a finite word $\theta=(\theta_0,\dots,\theta_k)$ such that the matrix $\rho(\theta):=\rho(\theta_k)\dots\rho(\theta_1)\in SL(n+1,\mathbb{R})$ is \emph{loxodromic};  
\item \label{item:homoclinic} there is an homoclinic point $z=(\overline{\theta}z_0z_1\cdots z_{\ell}\overline{\theta})\in\Sigma$ of $p_{\theta}$ such that, for any $1\leq k\leq n$ and any pair $F, G\subset\mathbb{R}^{n+1}$ of $\rho(\theta)$-invariant subspaces of complementary dimensions $k$ and $n+1-k$, one has $\rho(z_0\cdots z_{\ell})(F)\cap G=\{0\}$.  
\end{enumerate} 
As it turns out, Benoist \cite{B1} proved that $\rho(\pi_1(M))=\Gamma$ is \emph{Zariski-dense} in $SL(n+1,\mathbb{R})$ when $\Omega$ is not an ellipsoid. In particular, the existence of (many) periodic points $p_{\theta}$ such that $\rho(\theta)$ is loxodromic is assured following \cite{GM}. 

 The simultaneous transversality lemma of Abels--Margulis--Soifer \cite{AMS} (see also \cite[Section 2]{BS}) says that given an integer $N$ and a finite family of vectors $v_i$ and hyperplanes $H_i$, with 
 $1\leq i \leq N$, there exists $\gamma \in \Gamma$ such that, for every $(i, j) \in \{1, \cdots, N\}$, $\rho(\gamma)v_i \notin H_j$. 
 
In the sequel, we employ the simultaneous transversality lemma of Abels--Margulis--Soifer to construct the relevant homoclinic point. 
 Since $\Sigma$ is a shift of finite type,  we can fix a list of finite words $w_1,\dots, w_m$ such that for any finite word $(\zeta_1,\dots,\zeta_r)$ there are $1\leq i, j\leq m$ with $\overline{\theta}w_i\zeta_1\dots \zeta_rw_j\overline{\theta}\in\Sigma$. These words $w_1,\dots, w_m$ correspond to the transitions from $\theta$ to and from any other state of the graph associated to $\Sigma$. Let $(e_1, \cdots, e_{n+1})$ be a basis $\B$ of eigenvectors of the loxodromic element $\rho(p)$, $\F$ be the collection of vectors $\rho(w_j)(e_k)$, where $1 \leq j \leq m$ and 
$1 \leq k \leq  n+1$ and $\H$ the collection of hyperplane $\rho(w_i)^{-1}(H)$ where $H$ is an hyperplane generated by $n$ element of the basis $\B$ and $1 \leq i \leq m$.
We choose $\gamma$ satisfying the simultaneous transversality lemma of Abels-Margulis-Soifer for the collection of vector $\F$ and the collection of hyperplane $\H$. There exists $1\leq i, j\leq m$  such that the word $w_i\gamma w_j$ belongs to the language of $\Sigma$ and for each $v \in \F$ and $H \in \H$, 
$\rho(w_i\gamma w_j)(v) \notin H$. 

To conclude that $z=w_i\gamma w_j$ is the relevant homoclinic point, we have to prove that, for $F$ and $G$ of complementary dimensions both generated by elements of $\B$, we have 
$\rho(z)(F) \cap G = \{0\}.$ To return to the previous case, we consider exterior powers.  If $\textrm{dim}(F) = k \leq  \textrm{dim}(G) = n+1 -k$, we consider the action of $\Gamma$ on the exterior power $\Lambda^k(\mathbb{R}^{n+1})$.
Assume that up to permutation $F$ is generated by $(e_{i_1}, \cdots, e_{i_k})$ and $G$ is generated by $(e_{k+1}, \cdots, e_{n+1})$. If, by contradiction $\rho(z)(F) \cap G \neq \{0\}$, it means that $\rho(z)(e_1) \wedge \cdots \wedge \rho(z)(e_k)$ belongs to the hyperplane of $\Lambda^k(\mathbb{R}^{n+1})$ generated by all  the elementary wedges except $e_1 \wedge \cdots \wedge e_k$. Consequently, we repeat the previous argument and apply the simultaneous transversality lemma of Abels--Margulis--Soifer for the action of $\rho(z)$ on $\Lambda^k(\mathbb{R}^{n+1})$. This is possible since the transversality lemma can be applied  simultaneously for a finite number of irreducible representations (see for instance \cite[Section 2]{BS}). 
Thus $z$ satisfies the hypothesis of item \eqref{item:homoclinic} of the simplicity criterion. This ends the proof of Theorem \ref{t.B}.
\end{proof}

\begin{proof}[Proof of Corollary \ref{t.A}] Let $M=\Omega/\Gamma$ be a compact $n$-dimensional manifold with a real projective structure induced by a strictly convex subset $\Omega$ of $P(\mathbb{R}^{n+1})$ which is divisible by $\Gamma < SL(n+1,\mathbb{R})$. Our discussion in \S\ref{ss.Crampon} and \S\ref{ss.SRB-BM} says that the statement of Corollary \ref{t.A} is equivalent to the simplicity of the positive Lyapunov exponents of the Hilbert flow with respect to the SRB measure $\mu_{SRB}$ since its disintegrations along almost all unstable manifolds are absolutely continuous and, hence, the set of Oseledets regular points $(x,v)$ of $\mu_{SRB}$ leads to a subset of points $x^+$ of $\partial \Omega$ of full Lebesgue measure (compare with \cite[pp. 2380]{FK}, and see also \cite{FK1} for more details about $\mu_{SRB}$ for real projective structures).
\end{proof}

\begin{remark} The Hopf parametrization of the Bowen--Margulis measure is absolutely continuous with respect $\nu_o\otimes\nu_o$, where $\nu_o$ is a Patterson--Sullivan measure at any point $o\in\Omega$. If we replace $\mu_{SRB}$ by the Bowen--Margulis measure $\mu_{BM}$ in the discussion above, then we get the simplicity of the approximate regularity indices $\alpha_i(\mu_{BM})$, $i=1,\dots, n-1$, at $\nu_o$-almost every point of $\partial\Omega$ when $\Omega$ is not an ellipsoid. In particular, for $n>2$, the $\alpha_i(\mu_{BM})$'s are all equal to $1$ if and only if $\Omega$ is an ellipsoid (and this gives a positive answer to a question raised by Crampon in \cite{Cr3}). 
\end{remark}

\begin{remark} 
Now, we give some simple consequences of our main theorem. 

In dimension $n$, Crampon proved that the sum of the positive Lyapunov exponents of the Bowen--Margulis measure is equal to $n-1$, as in the hyperbolic case (see \cite{Cr2}). Similarly, when $n = 3$, that is, for $3$-manifolds, the positive Lyapunov exponents satisfy $\chi_1 + \chi_2 = 2$. Since the spectrum is simple, up to permutation, we have $\chi_1 > 1 > \chi_2$.

Crampon (see \cite{Cr2}) also proved that  the topological entropy $h_{top}$ of the geodesic flow is less or equal to $n-1$ with equality only in the hyperbolic case. Moreover, the topological entropy is achieved for the Bowen--Margulis measure. By the Ledrappier--Young formula, there exist $d_1,\dots, d_{n-1}$ positive real numbers (corresponding the dimensions of the Bowen--Margulis measure along its unstable directions) such that
$$h_{top} = d_1 \chi_1 + \dots + d_{n-1} \chi_{n-1}.$$ 
In the case of non hyperbolic surfaces, since $h_{top}<1$ and the Lyapunov exponent is $1$, one has $d_1 < 1$. Similarly, in the case of non hyperbolic $3$-manifolds, since $h_{top} < 2$, one has either $d_1<1$ or $d_2<1$.
\end{remark} 

\appendix
\section{Lyapunov exponents for some periodic orbits}\label{a.1}
In this Appendix, we calculate in a very explicit way the Lyapunov exponents of the parallel transport along a periodic orbit under mild assumptions. The method is slightly different from (but related to) the one in \cite{Cr2}.

Let  $g$ be an element of  $\Gamma$ and $\gamma_g$ the corresponding periodic orbit.  To understand the projective action of $g$ on $\Gamma$, it is useful to consider the linear action of $g$ on $\mathbb{R}^{n+1}$ together with the cone $\mathcal C \in \mathbb{R}^{n+1}$ with origin $O =(0,\dots,0)$ that defines $\Omega \subset P(\mathbb{R}^{n+1})$.  Just recall that each group element $g \in \Gamma$ preserving $\mathcal C$ is  bi-proximal (after Benoist \cite{B1}). In the sequel, we will assume that $g$ is diagonalizable. We can choose a normal base $ (e_0,\dots, e_j, \dots , e_n) $ on which the matrix associated to $ g$  is block diagonal with corresponding modulus of eigenvalues, possibly repeated with multiplicity,  $ (\lambda_0 > \dots\geq \lambda_i \geq \dots > \lambda_n)$. In these notations $ \mathbb{R}.e_0  = x^+$ and $ \mathbb{R}.e_n =x^-$ are the endpoints of the axis of the periodic geodesic on $\Omega$. The two hyperplans  tangent to the cone $\mathcal C $ are also $g$ invariant and their intersection is thus 
\begin{equation}\label{e.3.15}
 T_{x^+} \mathcal C\cap T_{x^-}\mathcal C = \oplus E_i .
\end{equation}

\begin{figure}[h!]
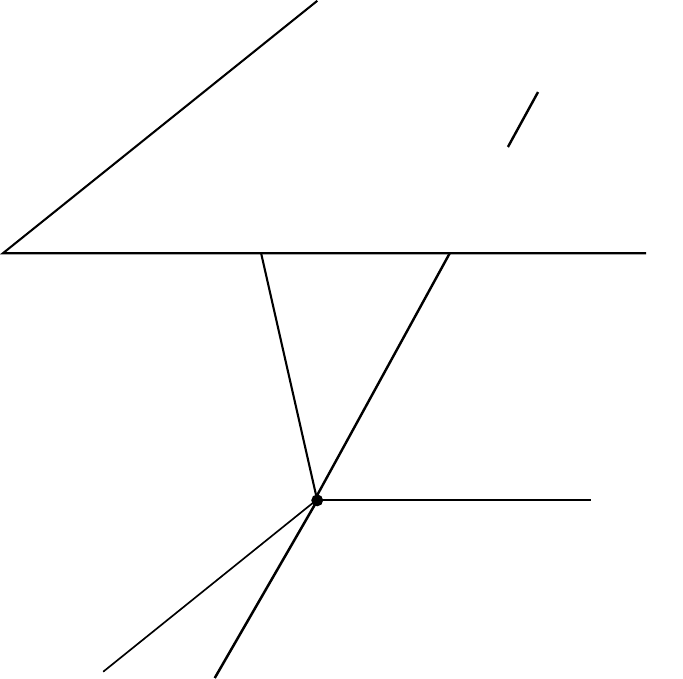\caption{Choice of coordinates adapted to periodic orbits.}
\end{figure}

If, as before, to identify real lines with projective points we choose a section of $\mathcal C$ containing the pair of points $ (x^+ \equiv O+ e_0, \, x^- \equiv O+e_n)$ and such that the two tangent affine  hyperplans at $x^+$ and $x^-$ are parallel, then any point $p \in \Omega $ is of the form 
\begin{align}\label{e.3.16}
 p = O + ue_0 + ve_n + \sum y_i,  && 
 u+v=1, &&
 y_i \in E_i. 
 \end{align}
In $ E_i$, the element $g$ acts as a compound homothetic by an orthogonal transformation $U_i$.\footnote{If the matrix is not diagonalizable, the matrix $U_i$ is not any more a composition of an homothety and an orthogonal transformation, it has to be composed with a unipotent matrix which does not change Lyapunov exponents but changes exact formulas. It is not clear whether non diagonalizable matrices could appear in this situation. }  
 %and has coordinates $( 0<u <1,  y_1, \cdots , y_n)$ \\ 
 The projective action of the  group element $g$ is thus of the form
  \begin{equation}\label{e.3.17}
  g(u,\dots, y_i,\dots) = \left( \frac {\lambda_0 . u }{\lambda_0. u + \lambda_n .v }, \cdots, 
  \frac{ \lambda_i. U_i. y_i}{\lambda_0. u + \lambda_n .v}, \cdots \right).
  \end{equation}
 Finally, the affine length in the sequel will be the one induced from the Euclidean structure on $\mathbb{R}^{n+1}$.
 \subsubsection{Length of periodic geodesics} 
 From the discussion above, the image of the lift of the periodic geodesic defined by $g$ is the set $\lbrace y_i=0, 1\leq  i\leq n ,  0<u <1 \rbrace$. For any point $ x= (u, 0,\dots,0)$, we get  
 \begin{align}\label{e.3.18}
 \vert x^- -x\vert = u \|e_0- e_n \| &, &  \vert x-x^+\vert = v \|e_0- e_n \| &,&  \vert x^- -x^+\vert=  \|e_0- e_n \|
 \end{align}
 
\begin{proposition}\label{p.3.24} 
The Hilbert length of the closed orbit associated to the group element $g$ is 
$$ L^H(g) = \frac{1}{2} \log \frac {\lambda_0(g)}{\lambda_n(g)}$$
\end{proposition}  
\begin{proof}
Let $x =(u,0,\dots,0)$ be a point of the axis. Using \eqref{e.3.16}, we get 
$$ gx = \left(  \frac {\lambda_0 . u }{\lambda_0. u + \lambda_n .v }, 0,\cdots,0\right)$$
and, thanks to \eqref{e.3.17},  one obtains 
\begin{align}\label{e.3.19}
\vert x^- - gx\vert =  \frac {\lambda_0 . u }{\lambda_0. u + \lambda_n .v } \|e_0- e_n \|  
\end{align}
\begin{align}\label{e.3.20}
\vert gx-x^+\vert =  \frac {\lambda_n . v }{\lambda_0. u + \lambda_n .v } \|e_0- e_n \|.
\end{align}
Since $ L^H(g) =d_{\Omega}( x , gx) = \frac{1}{2} \log \frac { \vert x-x^+\vert \ . \ \vert x^- -gx\vert}{\vert gx-x^+\vert \ . \ \vert x^- -x\vert }$ (by definition), the proof of the desired result is now complete.
\end{proof}

 \subsubsection{Simple calculation of the Lyapunov exponents of a closed orbit associated to a diagonalizable matrix}
  Consider  the action of the geodesic flow along a closed orbit $\gamma_g $ associated to a diagonalizable group element $g\in \Gamma$. To describe it explicitly, we have seen that our task is to understand, given a point $x$ on the closed geodesic $\gamma_g$, the parallel transport at time $T =  L^H(g)$ of a vector $ Z \in T_{x}S_{x^+}(x)$ tangent to the horosphere at $x$ with point at infinity $x^+$. According to Proposition \ref{p.3.21}, it remains to evaluate $g_T Z$ for $g_T= g^{-1}$. This amounts to compute the tangent map of the projective action of $g$ in the coordinates defined above. Similarly, the factor $ f(z,T)$ is determined by Lemma \ref{l.3.14} together with the formulas \eqref{e.3.18}, \eqref{e.3.19}, \eqref{e.3.20}. Using these facts, we will show below that: 
\begin{proposition} \label{pro:paralleltransport}
Let  $g\in \Gamma$ be a diagonalisable matrix and denote its eigenspaces by $E_i$. The Hilbert parallel transport $ \tau^H_T$ along the associated  closed orbit $ \gamma_g$ for time $T=L^H(g)$ preserves a splitting 
$$ THM = \mathbb{R}X \oplus\bigoplus_{i=1}^n T_iHM $$
where   $$ T_iHM= V_iHM \oplus {H_X} (V_iHM), \quad d\pi {H_X} (V_iHM) = E_i.$$
Moreover, for any $X_i \in T_iHM$, one has  
$$ \tau^H_T(X_i) = {\lambda_0^{\frac{1}{2}}}{\lambda_n^{\frac{1}{2}}}. \frac{U_i^{-1}.X_i}{\lambda_i}. $$
\end{proposition}
\begin{remark}
In the Riemannian case any $g\in \Gamma$ is such that $ \lambda_0.\lambda_n =1$ and $\lambda_i= 1$ then the parallel transport is an orthogonal transformation and the corresponding Lyapunov  exponents vanish.
\end{remark}
\begin{corollary}
The parallel Lyapunov exponent for $X_i \in T_iHM $ is 
$$ \eta_i = 1 + 2 \frac{\log \frac{\lambda_n}{\lambda_i}}{\log{\frac{\lambda_0}{\lambda_n}}} = -1 + 2 \frac{\log \frac{\lambda_0}{\lambda_i}}{\log{\frac{\lambda_0}{\lambda_n}}}$$
In particular 
$$ -1 <  \eta_i < 1$$ 
\end{corollary}
\begin{proof}[Proof of Proposition \ref{pro:paralleltransport}]
We have identified the tangent space to the horosphere at $x$ to the space $ \oplus E_i$ according to \eqref{e.3.15}. 
Hence, we have just to evaluate $dg^{-1}(Z_i)$: this is done by computing the differential of the map given in \eqref{e.3.17} and taking its inverse at the point $x_T$. In this way, we see that, for any $Z_i$, one has  
$$ dg^{-1}(Z_i)= \frac{\lambda_0. u + \lambda_n .v}{ \lambda_i} U_i ^{-1}. Z_i$$
As it was mentioned before, we can put this information together with \eqref{e.3.18}, \eqref{e.3.19}, \eqref{e.3.20} to get
$$ f((x,[v]), T ) = \frac{1}{2} .  \frac {1}{u. \exp(T) + v . \exp(-T)}.$$ 
In view of Proposition \ref{p.3.24}, we derive that 
$$ f((x,[v]), T ) = \frac{1}{2} . \lambda_n^{\frac{1}{2}} . \lambda_0^{\frac{1}{2}} .  \frac {1}{u. \lambda_0+ v . \lambda_n}.$$ 
This ends the proof of the desired proposition. 
\end{proof}

%\sloppy\printbibliography

\end{document}